\documentclass[12pt]{amsart}

\usepackage[T1]{fontenc}
\usepackage{lmodern}
\usepackage[margin=1in]{geometry}
\usepackage{amsmath,amssymb,amsthm,mathtools}
\usepackage{microtype}
\usepackage[hidelinks]{hyperref}
\usepackage{tikz}

\newtheorem{theorem}{Theorem}
\newtheorem{lemma}[theorem]{Lemma}
\numberwithin{equation}{section}
\newcommand{\R}{\mathbb R}
\newcommand{\Ms}{M_{\mathrm s}}
\newcommand{\Astr}[1]{A_{#1}^{\mathrm{str}}}
\newcommand{\achar}[2]{[#1]_{\Astr{#2}}}
\newcommand{\avg}[2]{\langle #1\rangle_{#2}}
\newcommand{\ind}{\mathbf 1}
\newcommand{\norm}[1]{\left\|#1\right\|}
\DeclareMathOperator{\dist}{dist}

\begin{document}
\title[Failure of the linear $A_2$ bound]
{Failure of the linear $A_2$ bound for the strong maximal operator}

\author[A.K. Lerner]{Andrei K. Lerner}
\address[A.K. Lerner]{Department of Mathematics,
Bar-Ilan University, 5290002 Ramat Gan, Israel}
\email{lernera@math.biu.ac.il}

\thanks{The author was supported by ISF grant no. 807/26.}

\begin{abstract}
We show that the strong maximal operator on $\R^2$ does not satisfy a
linear weighted $L^2$ estimate in terms of the rectangular $A_2$
characteristic. More precisely, we construct weights with arbitrarily
large characteristic for which the operator norm is bounded below by
$\achar{w}{2}\sqrt{\log\achar{w}{2}}$.
The proof is elementary. It uses an explicit two-parameter mass recurrence
on a product of geometric interval partitions, together with a direct
comparison of arbitrary rectangular averages with anchored averages.
\end{abstract}


\maketitle

\section{Introduction}
Let $M$ denote the Hardy--Littlewood maximal operator on $\R^n$, defined
using cubes with sides parallel to the coordinate axes. For $1<p<\infty$,
Buckley's theorem~\cite{B93} states that
\begin{equation}\label{eq:buckley}
 \norm{M}_{L^p(w)\to L^p(w)}
 \le C_{n,p}[w]_{A_p}^{1/(p-1)},\qquad w\in A_p,
\end{equation}
where
\[
 [w]_{A_p}:=\sup_Q
 \avg{w}{Q}\avg{w^{-1/(p-1)}}{Q}^{p-1},
 \qquad
 \avg{g}{Q}:=\frac1{|Q|}\int_Q g.
\]
This theorem is a starting point of
the theory of quantitative weighted norm inequalities. Importantly,
it gives the sharp dependence on the weight characteristic: the factor
$[w]_{A_p}^{1/(p-1)}$ cannot be replaced by a
function of $[w]_{A_p}$ of smaller order.

Consider now the strong maximal operator
\[
 \Ms f(x):=\sup_{R\ni x}\frac1{|R|}\int_R |f(y)|\,dy,
\]
where the supremum is taken over all bounded axis-parallel rectangles
$R\subset\R^n$. Its natural weight class is $\Astr{p}$, defined by
\[
 \achar{w}{p}:=\sup_R
 \avg{w}{R}\avg{w^{-1/(p-1)}}{R}^{p-1}<\infty.
\]
As in the cubic case, membership in this class characterizes the weighted
$L^p$ boundedness of~$\Ms$, see \cite{BK84}.

The quantitative picture is much less clear. For fixed $n\ge2$ and
$1<p<\infty$, let $\alpha_p$ be the infimum of the exponents $\alpha$
for which
\[
 \norm{\Ms}_{L^p(w)\to L^p(w)}
 \le C_{n,p,\alpha}\achar{w}{p}^{\alpha}
 \qquad\text{for all }w\in\Astr{p}.
\]
The elementary bounds are
\begin{equation}\label{eq:exponent-gap}
 \frac1{p-1}\le\alpha_p\le\frac n{p-1}.
\end{equation}

The lower bound follows from one-dimensional examples, while the upper
bound follows by dominating $\Ms$ by a composition of $n$
one-dimensional maximal operators and applying Buckley's theorem in each
coordinate. Despite the size of the gap in~\eqref{eq:exponent-gap},
to the best of our knowledge neither bound on the optimal power exponent
has been improved. In particular, it has remained open whether the direct
analogue of~\eqref{eq:buckley} holds for the strong maximal operator.
This question was explicitly raised in \cite[Section~5]{LPR17}.

In this note we show that the situation is different from the cubic case:
the estimate with the endpoint power $1/(p-1)$ fails already when
$n=2$ and $p=2$. For simplicity and clarity of exposition, we concentrate
on this case.

\begin{theorem}\label{thm:main}
There exist absolute constants $c,C>0$ such that, for every sufficiently
small $\theta>0$, there is a weight $w_\theta\in\Astr{2}(\R^2)$ satisfying
$$
 c\theta^{-1}\le\achar{w_\theta}{2}\le C\theta^{-1}
$$
and
$$
 \norm{\Ms}_{L^2(w_\theta)\to L^2(w_\theta)}
 \ge c\theta^{-1}\sqrt{\log\frac1\theta}.
$$
Consequently, along this family of weights,
$$
 \norm{\Ms}_{L^2(w_\theta)\to L^2(w_\theta)}
 \ge c\achar{w_\theta}{2}
          \sqrt{\log\achar{w_\theta}{2}}.
$$
In particular, no uniform linear bound in $\achar{w}{2}$ is possible.
\end{theorem}

The proof is based on an explicit two-parameter construction on the unit square. The reciprocal weight is defined by a simple mass recurrence on a family of nested rectangular cells. The recurrence is chosen so that the rectangular $A_2$ characteristic remains of the expected size, while a large hyperbolic family of rectangles contributes simultaneously to the lower bound for the strong maximal operator. This produces the logarithmic improvement over the linear estimate. An elementary averaging argument then allows us to pass from the special anchored rectangles appearing in the construction to arbitrary rectangles, and a reflection argument extends the example to $\mathbb R^2$.

\section{Proof of Theorem~\ref{thm:main}}\label{sec:proof}
\subsection{The partition and an averaging lemma}
Set $\Omega:=[0,1)^2$. We partition $\Omega$ into pairwise
disjoint cells as follows.
Fix an integer $D\ge1$, set $a=2^{-D}$, and partition $[0,1)$ into
\[
 I_0:=[0,a),\qquad
 I_r:=[2^{r-1}a,2^ra),\quad 1\le r\le D.
\]
For $0\le r,s\le D$, set $C_{r,s}:=I_r\times I_s$. Then
\[
 \Omega=\bigsqcup_{0\le r,s\le D} C_{r,s}.
\]

The cell
\[
 Q:=C_{0,0}=I_0\times I_0
\]
will play a distinguished role in the construction: it is the source
square used in the testing argument.

Write
\[
 P_r:=[0,2^ra)=\bigcup_{u=0}^r I_u,\qquad
 R_{r,s}:=P_r\times P_s.
\]
We will use the elementary relations
\begin{equation}\label{eq:geometry}
 \frac{|C_{u,v}|}{|R_{r,s}|}
 \le 2^{-(r-u)-(s-v)}\quad(u\le r,\ v\le s),
 \qquad
 |C_{r,s}|=\frac14|R_{r,s}|\quad(r,s\ge1).
\end{equation}

\begin{figure}[htbp]
\centering
\begin{tikzpicture}[scale=0.5,>=stealth]

\def\xA{0}
\def\xB{1}
\def\xC{2}
\def\xD{4}
\def\xE{8}
\def\xF{16}

\fill[blue!15] (\xA,\xA) rectangle (\xE,\xD);

\fill[red!45] (\xA,\xA) rectangle (\xB,\xB);

\fill[green!45] (\xD,\xC) rectangle (\xE,\xD);

\draw[thick] (\xA,\xA) -- (\xF,\xA); 
\draw[thick] (\xA,\xA) -- (\xA,\xF); 
\draw[thick,dashed] (\xF,\xA) -- (\xF,\xF); 
\draw[thick,dashed] (\xA,\xF) -- (\xF,\xF); 

\draw (\xB,\xA) -- (\xB,\xF);
\draw (\xC,\xA) -- (\xC,\xF);
\draw (\xD,\xA) -- (\xD,\xF);
\draw (\xE,\xA) -- (\xE,\xF);

\draw (\xA,\xB) -- (\xF,\xB);
\draw (\xA,\xC) -- (\xF,\xC);
\draw (\xA,\xD) -- (\xF,\xD);
\draw (\xA,\xE) -- (\xF,\xE);

\draw[blue, thick] (\xA,\xA) rectangle (\xE,\xD);

\draw[red, thick] (\xA,\xA) rectangle (\xB,\xB);
\draw[green!50!black, thick]
  (\xD,\xC) rectangle (\xE,\xD);

\node[below] at (0.5,0) {$I_0$};
\node[below] at (1.5,0) {$I_1$};
\node[below] at (3,0)   {$I_2$};
\node[below] at (6,0)   {$I_3$};
\node[below] at (12,0)  {$I_4$};

\node[left] at (0,0.5) {$I_0$};
\node[left] at (0,1.5) {$I_1$};
\node[left] at (0,3)   {$I_2$};
\node[left] at (0,6)   {$I_3$};
\node[left] at (0,12)  {$I_4$};

\node[
  red!70!black,
  anchor=north,
  inner sep=2pt
] (Qlabel) at (-1.15,-1.15) {$Q=C_{0,0}$};

\draw[
  ->,
  red!70!black,
  thick,
  shorten <=2pt
] (Qlabel.north) -- (0.5,0.5);

\node[green!50!black] at (6,3) {$C_{3,2}$};

\node[
  anchor=west,
  text=blue!70!black,
  fill=blue!15,
  fill opacity=1,
  text opacity=1,
  inner xsep=1.5pt,
  inner ysep=1pt
] at (1.15,1.5) {$R_{3,2}$};

\end{tikzpicture}

\caption{The partition of $\Omega$ for $D=4$.
The distinguished source square $Q=C_{0,0}$ is shown in red,
the anchored rectangle $R_{3,2}$ in blue, and its upper-right
cell $C_{3,2}$ in green.}
\label{fig:partition}
\end{figure}
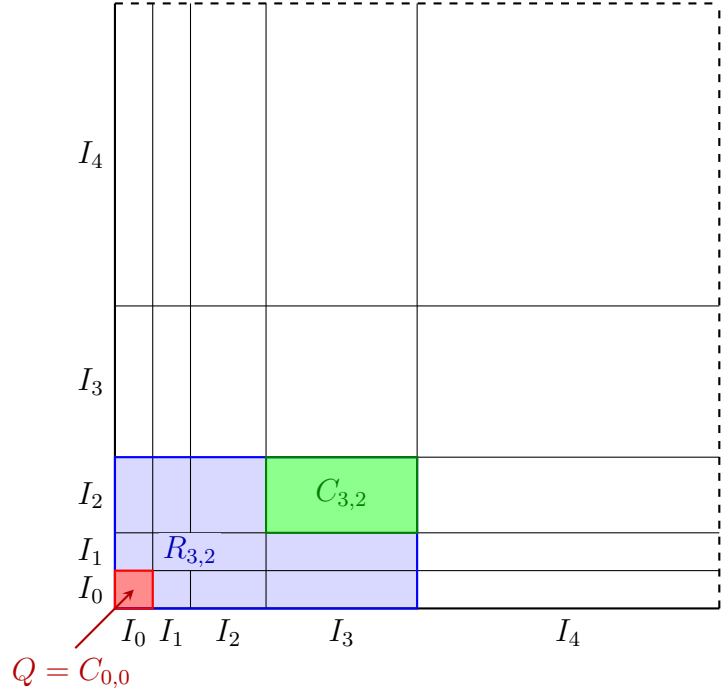

\begin{samepage}
\begin{lemma}\label{lem:averaging}
For every interval $J\subset[0,1)$, there is an index
$r$, depending only on $J$, such that
\begin{equation}\label{eq:interval-average}
 \avg{g}{J}\le4\avg{g}{P_r}
\end{equation}
for every nonnegative function $g$ constant on each $I_u$.
Consequently, every positive weight $v$ constant on the cells $C_{r,s}$
satisfies
\begin{equation}\label{eq:local-reduction}
 \sup_{J,K\subset[0,1)}
 \avg{v}{J\times K}\avg{v^{-1}}{J\times K}
 \le256\max_{0\le r,s\le D}
 \avg{v}{R_{r,s}}\avg{v^{-1}}{R_{r,s}},
\end{equation}
where the supremum is over arbitrary intervals $J,K$ of positive length.
\end{lemma}
\end{samepage}

\begin{proof}
Let $r$ be the largest index with $|J\cap I_r|>0$. We claim that
\begin{equation}\label{eq:coefficient-comparison}
 \frac{|J\cap I_u|}{|J|}
 \le4\frac{|I_u|}{|P_r|},\qquad 0\le u\le r.
\end{equation}
This is immediate for $r=0,1$. If $r\ge2$ and $J$ meets one of the intervals
$I_0,\ldots,I_{r-2}$ in positive measure, then it contains $I_{r-1}$ up to
endpoints. Hence $|J|\ge|I_{r-1}|=|P_r|/4$, which implies
\eqref{eq:coefficient-comparison}.
Otherwise, $J$ is contained, up to endpoints, in
$I_{r-1}\cup I_r$. For $u\le r-2$, the left-hand side of
\eqref{eq:coefficient-comparison} is zero. For
$u\in\{r-1,r\}$, we use
\[
 \frac{|J\cap I_u|}{|J|}
 \le 1
 \le 4\frac{|I_u|}{|P_r|},
\]
since $|I_{r-1}|=|P_r|/4$ and $|I_r|=|P_r|/2$.
Thus \eqref{eq:coefficient-comparison} holds in this
case as well.

Multiplying~\eqref{eq:coefficient-comparison} by the value
of $g$ on $I_u$ and summing proves~\eqref{eq:interval-average}.

Given $J\times K$, apply~\eqref{eq:interval-average} in each coordinate.
The resulting indices $r,s$ depend only on $J,K$, so for the same rectangle
$R_{r,s}$ we have
\[
 \avg{v}{J\times K}\le16\avg{v}{R_{r,s}},\qquad
 \avg{v^{-1}}{J\times K}\le16\avg{v^{-1}}{R_{r,s}}.
\]
Multiplication proves~\eqref{eq:local-reduction}.
\end{proof}

\subsection{The mass recurrence and the construction of the weight}
From now on let $0<\theta<1/256$ and set
\[
 c_0:=\frac1{32},\qquad
 D:=\left\lfloor\frac{c_0}{\theta}\right\rfloor.
\]
We use the preceding partition with this value of $D$. Our aim is
to construct a weight $v_\theta$ on~$\Omega$, constant on each
cell $C_{r,s}$. It is convenient to construct its reciprocal
$\sigma=v_\theta^{-1}$ first, by prescribing its masses on the cells.

The construction is based on the following rule: we normalize
$\sigma(Q)=1$ and require that every other cell carry a fixed
fraction $\theta$ of the mass of its corresponding anchored rectangle,
that is,
\[
 \sigma(C_{r,s})=\theta\,\sigma(R_{r,s}),
 \qquad (r,s)\ne(0,0).
\]
The numbers $m_{r,s}$ and $S_{r,s}$ introduced below will represent
the cell masses $\sigma(C_{r,s})$ and the rectangle masses
$\sigma(R_{r,s})$, respectively. Since $R_{r,s}$ is the disjoint
union of the cells $C_{u,v}$ with $u\le r$ and $v\le s$, these
numbers must satisfy
\[
 S_{r,s}=\sum_{u=0}^r\sum_{v=0}^s m_{u,v}.
\]
We first give an explicit formula for $S_{r,s}$ and
verify its compatibility with the prescribed cell masses.
The polynomials $T_{r,s}$ defined below are auxiliary quantities used to express
and verify this formula.

For nonnegative integers $r,s$, define
\begin{equation}\label{eq:S-definition}
 T_{r,s}:=\sum_{k\ge0}\binom{r}{k}\binom{s}{k}\theta^k,
 \qquad
 S_{r,s}:=(1-\theta)^{-(r+s)}T_{r,s},
\end{equation}
where binomial coefficients outside their usual range are zero.
Thus the sum is finite and $S_{r,s}>0$. Moreover,
$T_{r,0}=T_{0,s}=1$, so $S_{0,0}=1$ and
\begin{equation}\label{eq:axes}
 S_{r,0}=(1-\theta)^{-r},\qquad
 S_{0,s}=(1-\theta)^{-s}.
\end{equation}
For $r,s\ge1$, Pascal's identity gives
\begin{align*}
 &T_{r,s}-T_{r-1,s}-T_{r,s-1}+T_{r-1,s-1}\\
 &\quad=\sum_{k\ge0}
 \left(\binom{r}{k}-\binom{r-1}{k}\right)
 \left(\binom{s}{k}-\binom{s-1}{k}\right)\theta^k\\
 &\quad=\sum_{k\ge1}
 \binom{r-1}{k-1}\binom{s-1}{k-1}\theta^k
 =\theta T_{r-1,s-1}.
\end{align*}
Equivalently,
\[
 T_{r,s}
 =T_{r-1,s}+T_{r,s-1}-(1-\theta)T_{r-1,s-1}.
\]
Multiplying by $(1-\theta)^{-(r+s-1)}$ and using
\eqref{eq:S-definition}, we obtain
\begin{equation}\label{eq:recurrence}
 (1-\theta)S_{r,s}
 =S_{r-1,s}+S_{r,s-1}-S_{r-1,s-1}.
\end{equation}

For $0\le r,s\le D$, set
\begin{equation}\label{eq:cell-masses}
 m_{0,0}:=1,\qquad
 m_{r,s}:=\theta S_{r,s}\quad\text{if }(r,s)\ne(0,0).
\end{equation}
All these numbers are positive. We now verify that the prescribed
cell masses have the required cumulative masses.
By~\eqref{eq:axes} and~\eqref{eq:recurrence}, they are the mixed
differences of $S_{r,s}$:
\[
 m_{r,s}
 =S_{r,s}-S_{r-1,s}-S_{r,s-1}+S_{r-1,s-1},
 \qquad 0\le r,s\le D,
\]
with the convention that a term with a negative index is zero.
Telescoping in both indices gives
\begin{equation}\label{eq:mass-sum}
 \sum_{u=0}^r\sum_{v=0}^s m_{u,v}=S_{r,s}.
\end{equation}

We can therefore realize these numbers as masses of a positive
density: define
\[
 \sigma(x):=\frac{m_{r,s}}{|C_{r,s}|}
 \quad\text{for }x\in C_{r,s},
 \qquad
 v_\theta:=\sigma^{-1}\quad\text{on }\Omega.
\]
By~\eqref{eq:mass-sum} and the disjoint decomposition of $R_{r,s}$,
\[
 \sigma(R_{r,s})
 =\sum_{u=0}^r\sum_{v=0}^s \sigma(C_{u,v})
 =\sum_{u=0}^r\sum_{v=0}^s m_{u,v}
 =S_{r,s}.
\]
Also, $v_\theta$ has the constant value $|C_{r,s}|/m_{r,s}$ on
$C_{r,s}$. Hence
\begin{equation}\label{eq:weight-masses}
 \sigma(R_{r,s})=S_{r,s},\qquad
 v_\theta(C_{r,s})=\frac{|C_{r,s}|^2}{m_{r,s}}.
\end{equation}
In particular, $\sigma(Q)=1$, and the desired mass rule holds.

The purpose of this rule will become clear in the next two
subsections: it allows us to control the rectangular $A_2$
characteristic by a constant multiple of $\theta^{-1}$, while
the lower bound exploits the many anchored rectangles whose
$\sigma$-mass remains comparable to $\sigma(Q)=1$.

\subsection{The rectangular \texorpdfstring{$A_2$}{A2} estimate}

We first record a uniform growth bound. Another application of Pascal's
identity gives
\[
 T_{r+1,s}-T_{r,s}
 =\theta\sum_{j\ge0}\binom rj\binom s{j+1}\theta^j
 \le s\theta T_{r,s},
\]
because $\binom s{j+1}\le s\binom sj$. Since $s\theta\le c_0$,
\[
 \frac{S_{r+1,s}}{S_{r,s}}
 \le\frac{1+s\theta}{1-\theta}
 \le\frac{1+c_0}{1-\theta}<\frac32.
\]
The same estimate holds in the other coordinate. With $\kappa:=3/2$, we
therefore obtain
\begin{equation}\label{eq:growth}
 \frac{S_{r,s}}{S_{u,v}}
 \le\kappa^{(r-u)+(s-v)},\qquad
 0\le u\le r\le D,\quad 0\le v\le s\le D.
\end{equation}

Notice that $m_{u,v}\ge\theta S_{u,v}$ for every cell, including $(0,0)$.
Using~\eqref{eq:geometry}, \eqref{eq:weight-masses}, and~\eqref{eq:growth},
we have
\begin{align}
 \avg{v_\theta}{R_{r,s}}\avg{\sigma}{R_{r,s}}
 &=\frac{S_{r,s}}{|R_{r,s}|^2}
       \sum_{u\le r,\,v\le s}\frac{|C_{u,v}|^2}{m_{u,v}}\notag\\
 &\le\frac1\theta
       \sum_{u\le r,\,v\le s}
       \frac{S_{r,s}}{S_{u,v}}4^{-(r-u)-(s-v)}\notag\\
 &\le\frac1\theta\sum_{i,j\ge0}\left(\frac\kappa4\right)^{i+j}
 \lesssim\theta^{-1}.\label{eq:anchored-A2}
\end{align}
Lemma~\ref{lem:averaging} now gives the full local estimate
\begin{equation}\label{eq:local-A2}
 \sup_{R\subset \Omega}\avg{v_\theta}{R}\avg{v_\theta^{-1}}{R}
 \lesssim\theta^{-1},
\end{equation}
with the supremum over all axis-parallel rectangles in $\Omega$.

For the reverse inequality, $R_{1,0}$ is the disjoint union of the two
cells $C_{0,0}$ and $C_{1,0}$, each of area $a^2$. Their $\sigma$-masses
are $1$ and $\theta/(1-\theta)$, respectively. Hence
\begin{equation}\label{eq:A2-lower}
 \avg{v_\theta}{R_{1,0}}\avg{\sigma}{R_{1,0}}
 =\frac1{4\theta(1-\theta)}\gtrsim\theta^{-1}.
\end{equation}

\subsection{The lower bound for the maximal operator}

We first show that the masses $S_{r,s}$ remain uniformly bounded on a
large hyperbolic region.
Using
\[
 \binom{r}{k}\le \frac{r^k}{k!},
 \qquad
 \binom{s}{k}\le \frac{s^k}{k!},
\]
we obtain
\[
 T_{r,s}
 \le
 \sum_{k\ge0}\frac{(rs\theta)^k}{(k!)^2}
 \le
 \sum_{k\ge0}\frac{(rs\theta)^k}{k!}
 =
 e^{rs\theta}.
\]
Hence, whenever $r,s\le D$ and $rs\le D$, we have
\[
 rs\theta\le D\theta\le c_0,
\]
and therefore
\[
 T_{r,s}\le e^{c_0}.
\]
Also, since $-\log(1-\theta)\le2\theta$,
\[
 (1-\theta)^{-(r+s)}
 \le e^{2\theta(r+s)}
 \le e^{4c_0},
\]
because $r,s\le D$ and $D\theta\le c_0$.
Consequently,
\begin{equation}\label{eq:mass-hyperbola}
 S_{r,s}\le e^{5c_0}<2
 \qquad
 (1\le r,s\le D,\ rs\le D).
\end{equation}

We next count the pairs $(r,s)$ in this hyperbolic region. For each
$r=1,\dots,D$, the integers $s$ satisfying
\[
 1\le s\le D,
 \qquad
 rs\le D
\]
are exactly
\[
 1\le s\le \left\lfloor\frac Dr\right\rfloor.
\]
Thus
\[
 \#\{(r,s):1\le r,s\le D,\ rs\le D\}
 =
 \sum_{r=1}^D\left\lfloor\frac Dr\right\rfloor.
\]
Since $D/r\ge1$, we may use the elementary estimate
\[
 \lfloor x\rfloor\ge \frac{x}{2},
 \qquad x\ge1,
\]
and hence
\[
 \sum_{r=1}^D\left\lfloor\frac Dr\right\rfloor
 \ge
 \frac D2\sum_{r=1}^D\frac1r
 \gtrsim D\log D.
\]

Therefore, by \eqref{eq:mass-hyperbola},
\begin{equation}\label{eq:energy-lower}
 \sum_{r=1}^D\sum_{s=1}^D\frac1{S_{r,s}}\ge
 \sum_{\substack{1\le r,s\le D\\rs\le D}}
 \frac1{S_{r,s}}\gtrsim D\log D\gtrsim
 \theta^{-1}\log\frac1\theta.
\end{equation}

Take $f_\theta=\sigma\ind_Q$. By construction,
\begin{equation}\label{eq:input}
 \norm{f_\theta}_{L^2(v_\theta;\Omega)}^2
 =\int_Q\sigma^2v_\theta=\sigma(Q)=1.
\end{equation}
For $x\in C_{r,s}$, the rectangle $R_{r,s}$ contains both $x$ and $Q$.
Thus
\[
 \Ms f_\theta(x)\ge\frac1{|R_{r,s}|}\int_Q\sigma
 =\frac1{|R_{r,s}|}.
\]
For $r,s\ge1$, \eqref{eq:geometry}, \eqref{eq:cell-masses} and \eqref{eq:weight-masses} give
\[
 v_\theta(C_{r,s})=\frac{|R_{r,s}|^2}{16\theta S_{r,s}}.
\]
The output cells are pairwise disjoint, and therefore, by (\ref{eq:energy-lower}),
\begin{align}
 \int_{\Omega}(\Ms f_\theta)^2v_\theta
 &\ge\sum_{r,s=1}^D\frac{v_\theta(C_{r,s})}{|R_{r,s}|^2}\notag\\
 &=\frac1{16\theta}\sum_{r,s=1}^D\frac1{S_{r,s}}
 \gtrsim\theta^{-2}\log\frac1\theta.
 \label{eq:local-testing}
\end{align}

\subsection{Extension to the plane}
We finish with an elementary reflection argument. For a weight $v$ on
$[0,1]^2$, write
\[
 [v]_{2,[0,1]^2}
 :=
 \sup_{R\subset [0,1]^2}
 \avg{v}{R}\avg{v^{-1}}{R},
\]
where $R$ ranges over all axis-parallel rectangles of positive measure.

\begin{lemma}\label{lem:reflection}
Let
\[
 \pi(t):=\dist(t,2\mathbb Z)\in[0,1].
\]
If $v$ is a weight on $[0,1]^2$ with
$[v]_{2,[0,1]^2}<\infty$, then its reflected extension
\[
 V(x,y):=v(\pi(x),\pi(y))
\]
satisfies
\[
 [V]_{A_2^{\mathrm{str}}}
 \le 81 [v]_{2,[0,1]^2}.
\]
\end{lemma}

\begin{proof}
For every bounded interval $J\subset\mathbb R$ there is an interval
$J_*\subset[0,1]$, depending only on $J$, such that
\begin{equation}\label{eq:reflection-average}
 \frac1{|J|}\int_J g(\pi(t))\,dt
 \le
 \frac3{|J_*|}\int_{J_*}g(t)\,dt
\end{equation}
for every nonnegative measurable function $g$ on $[0,1]$.

Indeed, if $|J|<1$, then $J$ crosses at most one integer. Its image
$J_*=\pi(J)$ is an interval of length at most $|J|$, and the multiplicity
of $\pi$ on $J$ is at most two. Thus
\[
 \frac1{|J|}\int_J g(\pi(t))\,dt
 \le
 \frac2{|J|}\int_{J_*}g
 \le
 \frac2{|J_*|}\int_{J_*}g.
\]
If $|J|\ge1$, take $J_*=[0,1]$. The interval $J$ meets at most
$|J|+2$ unit intervals, and on each such interval $\pi$ is an isometry
onto $[0,1]$. Consequently,
\[
 \frac1{|J|}\int_J g(\pi(t))\,dt
 \le 3\int_0^1 g.
\]

Applying~\eqref{eq:reflection-average} successively in the two coordinates
shows, for the same rectangle $J_*\times K_*$, that
\[
 \avg{V}{J\times K}
 \le 9\avg{v}{J_*\times K_*},
 \qquad
 \avg{V^{-1}}{J\times K}
 \le 9\avg{v^{-1}}{J_*\times K_*}.
\]
Multiplication proves the lemma.
\end{proof}

We now apply the lemma to $v_\theta$. Recall that $v_\theta$ was originally
defined on
\[
 \Omega=[0,1)^2.
\]
Extend it to the closed square $[0,1]^2$ by assigning arbitrary positive
values on the top and right boundary edges. Since these edges have
Lebesgue measure zero, this extension does not affect any rectangular
averages, and hence
\[
 [v_\theta]_{2,[0,1]^2}
 =
 [v_\theta]_{2,\Omega}.
\]

Define the final weight on $\mathbb R^2$ by
\[
 w_\theta(x,y):=
 v_\theta(\pi(x),\pi(y)).
\]
By~\eqref{eq:local-A2} and Lemma~\ref{lem:reflection},
\[
 [w_\theta]_{A_2^{\mathrm{str}}}
 \lesssim \theta^{-1}.
\]
Since the extension agrees with
$v_\theta$ on $\Omega$, \eqref{eq:A2-lower} gives the matching lower bound.
Likewise, \eqref{eq:input} and~\eqref{eq:local-testing} yield
\[
 \norm{f_\theta}_{L^2(w_\theta)}=1,\qquad
 \norm{\Ms f_\theta}_{L^2(w_\theta)}
 \gtrsim\theta^{-1}\sqrt{\log\frac1\theta}.
\]
This completes the proof of Theorem~\ref{thm:main}.\hfill$\square$

\vskip 2mm
\noindent
{\bf AI disclosure statement.}
The author used ChatGPT (OpenAI) to assist in developing the construction and proof and in preparing the manuscript.
All arguments have been independently verified by the author, who takes full responsibility for the results and exposition.

\end{document}